\documentclass{amsart}%
\usepackage{amsfonts}%
\usepackage{amsmath}%
\usepackage{amssymb}%
\usepackage{graphicx}
\providecommand{\U}[1]{\protect\rule{.1in}{.1in}}
\newtheorem{theorem}{Theorem}
\theoremstyle{plain}
\newtheorem{acknowledgement}{Acknowledgement}

\newtheorem{definition}{Definition}

\newtheorem{lemma}{Lemma}

\numberwithin{equation}{section}

\def\s\varphi{\mathcal{\varphi}}

\begin{document}
\title[Compactness of  commutator of multilinear singular integral
operator]{Compactness of  commutators of multilinear singular integral
operators with non-smooth kernels}
\author{Rui Bu}
\address[A. One and A. Two]{Author One: Department of Mathematics, Zhejiang University, Hangzhou
310027, P.R. China}
\email[A. One]{burui0@163.com}
\author{Jiecheng Chen}
\curraddr[A. Two]{Department of Mathematics, Zhejiang Normal University, Jinhua
321004, P.R. China}
\email[A.~Two]{jcchen@zjnu.edu.cn}
\thanks{}
\date{}
\dedicatory{}
\begin{abstract}
  In this paper,  the behavior  for commutators of a class of bilinear singular integral operator associated  with non-smooth kernels on the  products of weighted Lebesgue spaces is considered. By some new maximal functions to control the commutators of bilinear singular integral operators and ${\rm CMO}$ functions,  compactness of the commutators is proved.
\end{abstract}
\maketitle

\section{Introduction}
In recent decades, the study of multilinear analysis becomes an active topic
in harmonic analysis. The first important work, among several pioneer papers,
is the famous work by Coifman and Meyer in  \cite{Coifman1}, \cite{Coifman2},
where they established a bilinear multiplier theorem on the Lebesgue spaces.
Note that a multilinear multiplier actually is a convolution operator.
Naturally one will study the non-convolution operator
\begin{equation}
T(f_{1},\ldots,f_{m})(x)=\int_{(\mathbf{R}^n)^{m}}K(x,y_{1},\ldots
,y_{m})f_{1}(y_{1})\cdots f_{m}(y_{m}){\rm d}y_{1}\cdots {\rm d} y_{m},
\end{equation}
where  $K(x,y_{1},\ldots,y_{m})$ is a locally integral function defined away
from the diagonal $x=y_{1}=\cdots=y_{m}$ in $(\mathbf{R}^n)^{m+1}$,
$x\notin\cap_{j=1}^{m} {\rm supp}\,f_{j}$ \ and  $f_{1},\ldots,f_{m}$ \ are bounded functions with compact supports. Precisely,  $T$: $\mathcal{S}%
(\mathbf{R}^n)\times\cdots\times\mathcal{S}(\mathbf{R}^n)\mapsto
\mathcal{S}^{^{\prime}}(\mathbf{R}^n)$  is an m-linear operator associated
with the kernel $K(x,y_{1},\ldots,y_{m})$. \ If there exist positive constants
$A$ and $\gamma\in(0,1]$ such that $K$ satisfies the size condition
\begin{equation}
|K(x,y_{1},\ldots,y_{m})|\leq\frac{A}{(|x-y_{1}|+\cdots+|x-y_{m}|)^{mn}}%
\end{equation}
for all $(x,y_{1},\ldots,y_{m})\in({\mathbf{R}^n})^{m+1}$ with $x\neq y_{j}$
\ for some $j\in\{1,2,\ldots,m\}$; and the smoothness condition
\begin{eqnarray}
&& |K(x,y_1,\ldots,y_j,\ldots,y_m)-K(x',y_1,\ldots,y_j,\ldots,y_m)|\\
  && \le \frac{A|x-x'|^{\gamma}}{(\sum^m_{i=1}|x-y_i|)^{mn+
  \gamma }},\notag
\end{eqnarray}
whenever $|x-x^{^{\prime}}|\leq\frac{1}{2}\max_{1\leq j\leq m}|x-y_{j}|$ and
also for each $j$,
 \begin{eqnarray}
 && |K(x,y_1,\ldots,y_j,\ldots,y_m)-K(x,y_1,\ldots,y^{'}_j,\ldots,y_m)|\\
 &&  \le \frac{A|y_j-y^{'}_j|^{\gamma}}{(\sum^m_{i=1}|x-y_i|)^{mn+
  \gamma}},\notag
\end{eqnarray}
whenever $|y_{j}-y_{j}^{^{\prime}}|\leq\frac{1}{2}\max_{1\leq j\leq m}%
|x-y_{j}|$, then we call that \ $K$ \ is a Calder\'{o}n-Zygmund kernel and
denote it by  $K\in m-CZK(A,\gamma).$ Also, $T$ is called the multilinear
Calder\'{o}n-Zygmund operator associated with the kernel \ $K.$ In
\cite{Grafakos5}, Grafakos and Torres established the multilinear $T1$
theorem, so that they obtained the strong type boundedness on products of
$L^{p}$ spaces and endpoint weak type estimates of operators $T$ associated
with kernels $K\in m-CZK(A,\gamma)$. Furthermore, the \ $A_{p}$ weights (see
Definition 1.2) on the operator $T$ and on the corresponding maximal operator
were considered in \cite{Grafakos4}. After then, the study of multilinear
Calder\'{o}n-Zygmund operator is fruitful. The reader can refer to
 \cite{Grafakos3}, \cite{Grafakos4}, \cite{Grafakos5}, \cite{Lerner}, \cite{Perez1}, \cite{Perez2},
\cite{Perez3} and the references therein.

However, there are some multilinear singular integral operators, including the
Calder\'{o}n commutator, whose kernels do not satisfy (1.4) (see
\cite{Duong2}). Here, the Calder\'{o}n commutator is defined by
\begin{equation}
\mathcal{C}_{m+1}(f,a_{1},\ldots,a_{m})(x)=\int_{\mathbf{R}}\frac{\prod
_{j=1}^{m}(A_{j}(x)-A_{j}(y))}{(x-y)^{m+1}}f(y){\rm d}y,
\end{equation}
where $A_{j}^{^{\prime}}=a_{j}$. \ In  \cite{Duong2}, the authors introduced a
class of multilinear singular integral operators whose kernels satisfy
``smoothness conditions'' weaker than those of the multilinear
Calder\'{o}n-Zygmund kernels, via the generalized approximation to the
identity. They first established a weak type estimate,  for $p_{1}%
,\ldots,p_{m+1}\in\lbrack1,\infty]\ \ $and$\ \ p\in(0,\infty)$ with $\frac
{1}{p}=\sum_{j=1}^{m+1}\frac{1}{p_{j}}$,
\[
\Vert\mathcal{C}_{m+1}(f,a_{1},\ldots,a_{m})\Vert_{L^{p,\infty}(\mathbf{R}%
)}\leq C\Vert f\Vert_{L^{p_{m+1}}(\mathbf{R})}\prod_{j=1}^{m}\Vert a_{j}%
\Vert_{L^{p_{k}}(\mathbf{R})}.
\]
If $\min_{1\leq j\leq m+1}p_{j}>1$, then the strong type estimate was also
established. The weighted estimates, including the multiple weights, of the
maximal Calder\'{o}n commutator were considered in \cite{Duong1} and
\cite{Grafakos3}. Moreover, there are a large amount of work related to
singular integral operators with non-smooth kernels. The reader may refer
\cite{Hu3}, \cite{Hu2} and \cite{Duong3}, among many interesting works.

In this article, we are interested in the compactness of the commutator of
multilinear singular integral operators with non-smooth kernels and
$\mathrm{CMO}$ functions, where $\mathrm{CMO}$ denotes the closure of
$C_{c}^{\infty}$ in the $\mathrm{BMO}$ topology. For the sake of convenience,
we will write out the case of compactness in a bilinear setting. In
particular, We will study the compactness of \ $T,$ where we assume that
\ $T$  is a bilinear singular integral operator associated with kernel $K$ in
the sense (1.1) and satisfying (1.2), and
\begin{itemize}
\item[\rm (i)] $T$ is bounded from
\begin{equation}
 L^{1}(\mathbf{R}^n)\times L^{1}(\mathbf{R}^n) \to L^{1/2,\infty
}(\mathbf{R}^n),
\end{equation}
\item[\rm (ii)] for $x,x^{^{\prime}},y_{1},y_{2}\in\mathbf{R}^n$ with $8|x-x^{^{\prime
}}|<\min_{1\leq j\leq2}|x-y_{j}|$,
\begin{equation}
|K(x,y_{1},y_{2})-K(x^{^{\prime}},y_{1},y_{2})|\leq\frac{D\tau^{\gamma}%
}{(|x-y_{1}|+|x-y_{2}|)^{2n+\gamma}},
\end{equation}
\end{itemize}
where \ $D$ is a constant and $\tau$ is a number such that $2|x-x^{^{\prime}%
}|<\tau$ and $4\tau<\min_{1\leq j\leq2}|x-y_{j}|$. \ It was pointed in
\cite{Hu4} that the above non-smooth kernel includes the non-smooth kernel
introduced by Doung et al. in \cite{Duong1}, \cite{Duong2}.  For
$b\in\mathrm{BMO}(\mathbf{R}^n)$, we consider commutators
\begin{align*}
&  T_{b}^{1}(f_{1},f_{2})=[b,T]_{1}(f_{1},f_{2})=T(bf_{1},f_{2})-bT(f_{1}%
,f_{2}),\\
&  T_{b}^{2}(f_{1},f_{2})=[b,T]_{2}(f_{1},f_{2})=T(f_{1},bf_{2})-bT(f_{1}%
,f_{2}).
\end{align*}
For $\vec{b}=(b_{1},b_{2})\in\mathrm{BMO}(\mathbf{R}^n)\times\mathrm{BMO}(\mathbf{R}^n)$, we consider the
iterated commutator
\[
T_{\vec{b}}(f_{1},f_{2})=[b_{2},[b_{1},T]_{1}]_{2}(f_{1},f_{2}),
\]
and, in the sense of (1.1),
\begin{align*}
&  [b,T]_{1}(f_{1},f_{2})(x)=\int_{\mathbf{R}^n}\int_{\mathbf{R}^n%
}K(x,y_{1},y_{2})(b(y_{1})-b(x))f_{1}(y_{1})f_{2}(y_{2}){\rm d}y_{1}{\rm d} y_{2},\\
&  [b,T]_{2}(f_{1},f_{2})(x)=\int_{\mathbf{R}^n}\int_{\mathbf{R}^n%
}K(x,y_{1},y_{2})(b(y_{2})-b(x))f_{1}(y_{1})f_{2}(y_{2}){\rm d} y_{1}{\rm d} y_{2},\\
&  T_{\vec{b}}(f_{1},f_{2})(x)=\int_{\mathbf{R}^n}\int_{\mathbf{R}^n%
}K(x,y_{1},y_{2})(b_{1}(y_{1})-b_{1}(x))(b_{2}(y_{2})-b_{2}(x))f_{1}%
(y_{1})f_{2}(y_{2}){\rm d} \vec{y}.
\end{align*}

Our aim is to obtain the compactness of above commutators. Before stating our
results, we briefly describe the background and our motivation. In
\cite{Calderon}, Calder\'{o}n first proposed the concept of compactness in the
multilinear setting and B\'{e}nyi and Torres put forward an equivalent one in
\cite{Benyi3}. B\'{e}nyi and Torres extended the result of compactness for
linear singular integrals by Uchiyama \cite{Uchiyama} to the bilinear setting
and obtained that  $[b,T]_{1}$, $[b,T]_{2}$, $[b_{2},[b_{1},T]_{1}]_{2}$ are
compact bilinear operators from $L^{p_{1}}(\mathbf{R}^n)\times L^{p_{2}%
}(\mathbf{R}^n)$ to $L^{p}(\mathbf{R}^n)$ when $b,b_{1},b_{2}%
\in\mathrm{CMO}(\mathbf{R}^n)$, $1<p_{1},p_{2}<\infty$ and $1/{p_{1}}+1/{p_{2}}=1/p\leq1$.
Recently, Clop and Cruz \cite{Clop} considered the compactness of the linear
commutator on weighted spaces. For the bilinear case, B\'{e}nyi et al.
\cite{Benyi1} extended the result of \cite{Benyi3} to the weighted case, and
they obtained that all  $[b,T]_{1}$, $[b,T]_{2}$, $[b_{2},[b_{1},T]_{1}]_{2}$
are compact operators from $L^{p_{1}}(w_{1})\times L^{p_{2}}(w_{2})$ to
$L^{p}(\nu_{\vec{w}})$ when $1<p_{1},p_{2}<\infty$, $1/{p_{1}}+1/{p_{2}%
}=1/p<1$, $\vec{w}\in A_{p}(\mathbf{R}^n)\times A_{p}(\mathbf{R}^n)$ and $b,b_{1},b_{2}\in\mathrm{CMO}(\mathbf{R}^n)$.
We note that in \cite{Benyi1}, $T$ is a Calder\'{o}n-Zygmund operator with
smooth kernel. Hence, in this article, we will consider the same compactness
for these commutators by assuming \ $T$ \ is an operator associated with
non-smooth kernel. Although we will adopt the concept of compactness proposed
in \cite{Benyi3} (The reader can refer to \cite{Benyi3} and \cite{Yosida} for more properties of
compact and precompact) and some basic ideas used in \cite{Benyi3}, \cite{Chen1}, \cite{Chen2},  \cite{Hu4},
\cite{Lerner} and \cite{Perez3}, our proof meet some special
difficulties so that some new ideas and estimates must be bought in.
Particularly, some specific maximal functions will be involved.

We denote the closed ball of radius $r$ centered at the origin in the normed
space $X$ as $B_{r,X}=\{x\in X:\Vert x\Vert\leq r\}$.

\begin{definition}
A bilinear operator $T: X\times Y\mapsto Z$ is called compact if
$T(B_{1,X}\times B_{1,Y})$ is precompact in $Z$.
\end{definition}

\begin{definition}
A weight $w$ belongs to the class $A_{p}$, $1<p<\infty,$ if
\[
\sup_{Q}\bigg(\frac{1}{|Q|}\int_{Q}w(y){\rm d}y\bigg)\bigg(\frac{1}{|Q|}\int_{Q}w(y)^{1-p^{^{\prime
}}}{\rm d} y\bigg)^{p-1}<\infty.
\]
A weight $w$ belongs to the class $A_{1}$ if there is a constant $C$ such
that
\[
\frac{1}{|Q|}\int_{Q}w(y){\rm d}y\leq C\inf_{x\in Q}w(x).
\]

\end{definition}

\begin{definition}
Let $\vec{p}=(p_{1},p_{2})$ and $1/p=1/{p_{1}}+1/{p_{2}}$ with $1\leq
p_{1},p_{2}<\infty$. Given $\vec{w}=(w_{1},w_{2})$, set $\nu_{\vec{w}}%
=\prod_{j=1}^{2}w_{j}^{p/p_{j}}$. We say that $\vec{w}$ satisfies the
$A_{\vec{p}}$ condition if
\[
\sup_{Q}\bigg(\frac{1}{|Q|}\int_{Q}\nu_{\vec{w}}\bigg)^{1/p}\prod_{j=1}^{2}\bigg(\frac
{1}{|Q|}\int_{Q}w_{j}^{1-p_{j}^{^{\prime}}}\bigg)^{1/p_{j}^{^{\prime}}}<\infty.
\]
Here,  $\Big(\frac{1}{|Q|}\int_{Q}w_{j}^{1-p_{j}^{^{\prime}}}\Big)^{1/p_{j}^{^{\prime
}}}$ is understood as $(\inf_{Q}w_{j})^{-1},$when $p_{j}=1$.
\end{definition}

The following two theorems are our main results:

\begin{theorem}\label{T1}
Let $T$ be a bilinear operator satisfying condition (1.6) and its kernel $K$
satisfy (1.2), (1.7). Assume  $b\in\mathrm{CMO}(\mathbf{R}^n)$, $p_{1}%
,p_{2}\in(1,\infty)$, $p\in(1,\infty)$ such that $1/p=1/p_{1}+1/p_{2}$ and
$\vec{w}=(w_{1},w_{2})\in A_{\vec{p}}(\mathbf{R}^n)$ such that $\nu
_{\vec{w}}\in A_{p}(\mathbf{R}^n)$. Then $[b,T]_{1}$, $[b,T]_{2}$ are
compact from $L^{p_{1}}(\mathbf{R}^n,w_{1})\times L^{p_{2}}(\mathbf{R}^n,w_{2})$ to $L^{p}(\mathbf{R}^n,\nu_{\vec{w}})$.
\end{theorem}

In order to prove Theorem \ref{T1}, we need the following result which has
independent interest.

\begin{theorem}\label{T2}
Let $T$ be a bilinear operator satisfying condition (1.6) and its kernel $K$
satisfy (1.2), (1.7). Assume  $b\in\mathrm{BMO}(\mathbf{R}^n)$, $p_{1}%
,p_{2}\in(1,\infty)$, $p\in(0,\infty)$ such that $1/p=1/p_{1}+1/p_{2}$,
$\vec{w}=(w_{1},w_{2})\in A_{\vec{p}}(\mathbf{R}^n)$. Then
\[
\Vert\lbrack b,T]_{1}(f_{1},f_{2})\Vert_{L^{p}(\nu_{\vec{w}})},\Vert\lbrack
b,T]_{2}(f_1,f_2)\Vert_{L^{p}(\nu_{\vec{w}})}\leq C\Vert b\Vert_{\mathrm{BMO}(\mathbf{R}^n)%
}\Vert f_{1}\Vert_{L^{p_{1}}(w_{1})}\Vert f_{2}\Vert_{L^{p_{2}}(w_{2})}.
\]

\end{theorem}

\noindent\textbf{Remark 1.1 } Theorem \ref{T1} and \ref{T2} are also true for the
iterated commutator $[b_{2},[b_{1},T]_{1}]_{2}$, and their proofs are similar
to the proof of Theorem \ref{T1} and \ref{T2}. We leave the detail to the interested reader.

We make some conventions. In this paper, we always denote a positive constant
by $C$ which is independent of the main parameters and its value may differ
from line to line. For a measurable set $E$, $\chi_{E}$ denotes its
characteristic function. For a fixed $p$ with $p\in\lbrack1,\infty)$,
$p^{^{\prime}}$ denotes the dual index of $p$. We also denote  $\vec{f}%
=(f_{1},\cdots,f_{m})$ \ with scalar functions \ $f_{j}~(j=1,2,...,m)$. \ Given
$\alpha>0$ and a cube $Q$, $\ell(Q)$ denotes the side length of $Q$, and
$\alpha Q$ denotes the cube which is the same center as $Q$ and $\ell(\alpha
Q)=\alpha\ell(Q)$.  $f_{Q}$ denotes the average of $\ f$ over $Q$.
Let $M$ be the standard Hardy-Littlewood maximal operator.  For $0< \delta <\infty $, $M_{\delta}$ is the maximal operator defined by
\begin{equation*}
 M_{\delta}f(x)=M(|f|^{\delta})^{1/\delta}(x)=\left(\sup_{Q\ni x}\frac{1}{|Q|} \int_Q |f(y)|^{\delta}{\rm d}y\right)^{1/\delta},
\end{equation*}
  $M^{\#}$  is the sharp maximal operator defined by Fefferman and Stein \cite{fefferman},
\begin{equation*}
M^{\#} f(x)=\sup_{Q\ni x} \inf_c \frac{1}{|Q|}\int_Q |f(y)-c|{\rm d}y\approx \sup_{Q\ni x}  \frac{1}{|Q|}\int_Q |f(y)-f_Q|{\rm d}y,
\end{equation*}
and
\[M^{\#}_{\delta}f(x)=M^{\#}(|f|^{\delta})^{1/\delta}(x).\]
It is known that, when  $0<p,\delta<\infty$, $w\in A_{\infty}(\mathbf{R}^n)$, there exists a  $C>0$ such that
\begin{equation}
 \int_{\mathbf{R}^n}(M_{\delta}f(x))^pw(x){\rm d}x \le C\int_{\mathbf{R}^n}(M^{\#}_{\delta}f(x))^p w(x){\rm d}x
\end{equation}
for any function $f$ for which the left-hand side is finite.

\section{A multilinear maximal operator}
We need some basis facts about  Orlicz spaces, for more information about these spaces the reader may consult \cite{Rao}.
For  $\Phi(t)=t(1+\log^+t)$  and a cube $Q$ in $\mathbf{R}^n$, we define
\begin{equation}
\|f\|_{L(\log L),Q}=\inf\{\lambda >0 :\frac{1}{|Q|}\int_Q \Phi\left(\frac{|f(x)|}{\lambda}\right) {\rm d} x\le 1\}.
\end{equation}
It is obvious that $\|f\|_{L(\log L),Q} >1$ if and only if $\frac{1}{|Q|}\int_Q \Phi(|f(x)|) {\rm d}x >1$. The generalized H\"older  inequality in Orlicz space together with the John-Nirenberg inequality imply that
\begin{equation}
\frac{1}{|Q|}\int_Q|b(y)-b_Q|f(y){\rm d}y \le C\|b\|_{{\rm BMO}(\mathbf{R}^n)}\|f\|_{L(\log L),Q}.
\end{equation}
 Define the maximal operator $\mathcal{M}_{L(\log L)}$ by
\begin{equation}
\mathcal{M}_{L(\log L)}(\vec{f})(x)=\sup_{Q\ni x}\prod^2_{j=1}\|f_j\|_{L(\log L),Q},
\end{equation}
 where the supremum is taken over all the cubes containing $x$. The following  boundedness  for $\mathcal{M}_{L(\log L)}(\vec{f})$ was proved in \cite{Lerner}.
 \begin{lemma}\label{lemma1}
If $1<p_1,p_2<\infty$, $\frac{1}{p}=\sum^2_{j=1} \frac{1}{p_j}$,  and $\vec{w}=(w_1,w_2)\in A_{\vec{p}}(\mathbf{R}^{2n})$,
 then  $\mathcal{M}_{L(\log L)}(\vec{f})$  is bounded from $L^{p_1}(w_1) \times L^{p_2}(w_2)$ to $L^{p}(\nu_{\vec{w}})$.
\end{lemma}
Lemma \ref{lemma1}  is  helpful in the proof of Theorem \ref{T2}.  Besides this maximal operator, we need several other maximal operators in the following.

 In \cite{Lerner}, a  maximal function $\mathcal{M}(\vec{f})$ was introduced, and its definition is
\[ \mathcal{M}(\vec{f} )(x)= \sup_{Q\ni x} \prod^2_{j=1}\left(\frac{1}{|Q|}\int_Q|f_j(y_j)| {\rm d}y_j\right),\]
where  the supremum is taken over all cubes $Q$ containing $x$. The boundedness of $\mathcal{M}(\vec{f})$ on weighted spaces was considered in [22, Theorem 3.3].

  Furthermore, Grafakos, Liu, and Yang \cite{Grafakos3} introduced  some new  multilinear  maximal operators:
\begin{eqnarray*}
&&\mathcal{M}_{2,1} (\vec{f } )(x)=\sup_{Q\ni x} \sum^{\infty}_{k=0} 2^{-kn}\left( \frac{1}{|Q|} \int_Q |f_1(y_1)|{\rm d}y_1\right) \left( \frac{1}{|2^{k}Q|} \int_{2^kQ} |f_2(y_2)|{\rm d}y_2\right) ,\\
&&\mathcal{M}_{2,2} (\vec{f } )(x)=\sup_{Q\ni x} \sum^{\infty}_{k=0} 2^{-kn}\left( \frac{1}{|Q|} \int_Q |f_2(y_2)|{\rm d}y_2\right) \left( \frac{1}{|2^{k}Q|} \int_{2^kQ} |f_1(y_1)|{\rm d}y_1\right),
\end{eqnarray*}
where $\vec{f}=(f_1,f_2)$ and each $f_j$ ($j\in\{1,2\}$) is a  locally integrable function.
 The following boundedness of $\mathcal{M}_{2,1}$ and $\mathcal{M}_{2,2}$ were proved in \cite{Grafakos3}.
\begin{lemma}\label{lemma2}
Let $1<p_1,p_2<\infty$, $\frac{1}{p}=\sum^2_{j=1} \frac{1}{p_j}$,  and $\vec{w}=(w_1,w_2)\in A_{\vec{p}}(\mathbf{R}^{2n})$. Then
  $\mathcal{M}_{2,1}$ and $\mathcal{M}_{2,2}$ are bounded from $L^{p_1}(w_1) \times L^{p_2}(w_2)$ to $L^{p}(\nu_{\vec{w}})$.
\end{lemma}

In addition, Hu \cite{Hu1} introduced another kind of  bilinear  maximal operators $\mathcal{M}^1_{\beta}$ and $\mathcal{M}^2_{\beta}$ which was defined by
\begin{eqnarray*}
&&\mathcal{M}^1_{\beta} (\vec{f})(x)=\sup_{Q\ni x} \frac{1}{|Q|} \int_Q |f_1(y_1)|{\rm d}y_1\sum^{\infty}_{k=1} 2^{-kn}2^{k\beta}\frac{1}{|2^kQ|} \int_{2^kQ} |f_2(y_2)|{\rm d}y_2 ,\\
&&\mathcal{M}^2_{\beta} (\vec{f})(x)=\sup_{Q\ni x} \frac{1}{|Q|} \int_Q |f_2(y_2)|{\rm d}y_2\sum^{\infty}_{k=1} 2^{-kn}2^{k\beta}\frac{1}{|2^kQ|} \int_{2^kQ} |f_1(y_1)|{\rm d}y_1,
\end{eqnarray*}
where $\beta\in\mathbf{R}$ and the supremum is taken over all cubes $Q$ containing $x$. As it is well known,  a weight $w\in A_{\infty}(\mathbf{R}^n)$  implies that there exists  a  \,$\theta\in(0,1)$ such that for all cubes $Q$ and any set  $E\subset Q$,
\begin{equation}
\frac{w(E)}{w(Q)} \le C\left(\frac{|E|}{|Q|}\right)^{\theta}.
\end{equation}
For a fixed $\theta\in(0,1)$, set
\[R_{\theta}=\{w\in A_{\infty}(\mathbf{R}^n):w ~{\rm satisfies}~ (2.4)\}.\]
 In \cite{Hu1}, the following boundedness of $\mathcal{M}^1_{\beta}$ and  $\mathcal{M}^2_{\beta}$ were proved.
\begin{lemma}\label{lemma3}
Let $1<p_1,p_2<\infty$, $\frac{1}{p}=\sum^2_{j=1} \frac{1}{p_j}$, $\vec{w}=(w_1,w_2)\in A_{\vec{p}}(\mathbf{R}^{2n})$ and $\nu_{\vec{w}}\in R_{\theta}$ for some $\theta$ such that $\beta<n\theta\min\{1/p_1,1/p_2\}$. Then $\mathcal{M}^1_{\beta}$ and  $\mathcal{M}^2_{\beta}$ are bounded from $L^{p_1}(w_1) \times L^{p_2}(w_2)$ to $L^{p}(\nu_{\vec{w}})$.
\end{lemma}

\section{Proof of Theorem 2}
   The proof of  Theorem \ref{T2} will  depend on some pointwise estimates using sharp maximal functions. The pointwise estimates are the following:
\begin{lemma}\label{lemma4}
Let  $T$ be a bilinear  operator satisfying condition (1.6) and its kernel  $K$  satisfy (1.2), (1.7), if $0<\delta < \frac{1}{2}$. Then for all $\vec{f}$ in any product of $\L^{p_j}(\mathbf{R}^n)$ spaces with $1\le p_j < \infty$
\begin{eqnarray*}
M^{\#}_{\delta}(T(\vec{f}))(x) \le C \mathcal{M}(\vec{f})(x) + C \sum^2_{i=1}\mathcal{M}_{2,i}(\vec{f})(x).
\end{eqnarray*}
\end{lemma}
The proof of this Lemma uses some ideas of [22, Theorem 3.2] and the following Lemma \ref{lemma5}. Its proof   is not hard, so we omit.

\begin{lemma}\label{lemma5}
Let  $T$ be a bilinear  operator satisfying condition (1.6) and its kernel $K$ satisfy (1.2), (1.7). If $T_b^1$, $T_b^2$ be  commutators with $b\in {\rm BMO}(\mathbf{R}^n)$. For $0< \delta <\epsilon$ with $0<\delta<1/2$ let $r>1$ and $0<\beta<n$. Then, there exists a constant $C>0$, depending on $\delta$ and $\epsilon$, such that
\begin{eqnarray*}
&&\sum^2_{i=1}M^{\#}_{\delta}(T_b^i(\vec{f} ))(x)\le C \|b\|_{\rm BMO(\mathbf{R}^n)} \Big(\mathcal{M}_{L(\log L)}(\vec{f})(x)\\
&&\quad\quad+M_{\epsilon}(T(\vec{f}))(x)+\sum^2_{i=1} \{\mathcal{M}_{\beta}^i(f_1^r,f_2^r)(x)\}^{1/r} \Big)
\end{eqnarray*}
for all $\vec{f}=(f_1,f_2)$ of  bounded functions with compact support.
\end{lemma}

\begin{proof}
We only write out the proof of $M^{\#}_{\delta}(T_b^1(\vec{f}))(x)$, the other can be obtained by symmetry. In our proof we wil use some ideas of  \cite{Perez3}. For a fixed $x\in\mathbf{R}^n$, a cube $Q$ centered at $x$ and  constants $c, \lambda$, because $0<\delta<1/2$,
\begin{eqnarray*}
&&\left(\frac{1}{|Q|}\int_Q \big||T_b^1(\vec{f})(z)|^{\delta}-|c|^{\delta} \big|{\rm d}z \right)^{1/\delta}
\le \left(\frac{1}{|Q|}\int_Q |T_b^1(\vec{f})(z)-c|^{\delta} {\rm d}z \right)^{1/\delta}  \\
&&\le\left(\frac{C}{|Q|}\int_Q |(b(z)-\lambda)T(f_1,f_2)(z)-T((b(z)-\lambda)f_1,f_2)(z)-c|^{\delta} {\rm d}z \right)^{1/\delta} \\
&& \le\left(\frac{C}{|Q|}\int_Q |(b(z)-\lambda)T(f_1,f
_2)(z)|^{\delta} {\rm d}z \right)^{1/\delta} \\
&&\quad+ \left(\frac{C}{|Q|}\int_Q |T((b(z)-\lambda)f_1,f_2)(z)-c|^{\delta} {\rm d}z \right)^{1/\delta} \\
&&~=I_1+I_2.
\end{eqnarray*}
Let $Q^{*}=8^nQ$,  $\lambda=b_{Q^{*}}$.  The proof of the first part is the same as [25, Theorem 3.1]. Therefore, we omit the proof, and from [25, Theorem 3.1], we  obtain that
\begin{eqnarray*}
&&I_1 \le C\|b\|_{{\rm BMO}(\mathbf{R}^n)} M_{\epsilon}(T(f_1,f_2))(x).
\end{eqnarray*}
Hence, we only consider the second part $I_2$. We decompose  $f_1,f_2$ as $f_1=f_1^1+f_1^2=f_1(x)\chi_{Q^{*}}+f_1(x)\chi_{\mathbf{R}^n\setminus Q^{*}}$,  $f_2=f_2^1+f_2^2=f_2(x)\chi_{Q^{*}}+f_2(x)\chi_{\mathbf{R}^n\setminus Q^{*}}$. Let $c=c_1+c_2+c_3$ and
\begin{eqnarray*}
&&c_1=T((b-\lambda)f_1^1,f_2^2)(x),  \\
&&c_2=T((b-\lambda)f_1^2,f_2^1)(x), \\
&&c_3=T((b-\lambda)f_1^2,f_2^2)(x).
\end{eqnarray*}
Therefore,
\begin{eqnarray*}
I_2&\le &\left(\frac{C}{|Q|}\int_Q |T((b-\lambda)f_1^1,f_2^1)(z)|^{\delta} {\rm d}z \right)^{1/\delta} \\
&& + \left(\frac{C}{|Q|}\int_Q |T((b-\lambda)f_1^1,f_2^2)(z)-c_1|^{\delta} {\rm d}z \right)^{1/\delta}\\
&& + \left(\frac{C}{|Q|}\int_Q |T((b-\lambda)f_1^2,f_2^1)(z)-c_2|^{\delta} {\rm d}z \right)^{1/\delta} \\
&& + \left(\frac{C}{|Q|}\int_Q |T((b-\lambda)f_1^2,f_2^2)(z)-c_3|^{\delta} {\rm d}z \right)^{1/\delta}. \\
&=& I_2^1 + I_2^2 + I_2^3 +I_2^4.
\end{eqnarray*}
We choose $1<q<1/(2\delta)$. By H\"older's inequality, Kolmogorov inequality and the fact that $T$ satisfies condition (1.6), we get

\begin{eqnarray*}
I_2^1&\le& \left(\frac{C}{|Q|}\int_Q |T((b-\lambda)f_1^1,f_2^1)(z)|^{q\delta} {\rm d}z \right)^{1/q\delta}\\
&\le&C\|T((b-\lambda)f_1^1,f_2^1)\|_{L^{1/2,\infty}(Q,\frac{dz}{|Q|})} \\
&\le&C\left(\frac{1}{|Q|}\int_Q |(b(z)-\lambda)f_1^1(z)| {\rm d}z \right)\left(\frac{1}{|Q|}\int_Q |f_2^1(z)| {\rm d}z \right)\\
&\le& C\|b\|_{\rm BMO(\mathbf{R}^n)}\|f_1\|_{L(\log L),Q}\|f_2\|_{L(\log L),Q}  \\
&\le& C\|b\|_{\rm BMO(\mathbf{R}^n)} \mathcal{M}_{L(\log L)}(f_1,f_2)(x).
\end{eqnarray*}
Next,  we  estimate $I_2^2$ by generalized Jensen's inequality,
\begin{eqnarray*}
&&|T((b-\lambda)f_1^1,f_2^2)(z)-T((b-\lambda)f_1^1,f_2^2)(x)|\\
&&\le\int_{\mathbf{R}^{2n}} \frac{C}{(|z-y_1|+|z-y_2|)^{2n}} |(b-\lambda)f_1^1(y_1)||f_2^2(y_2)| {\rm d}y_2{\rm d}y_1   \\
&&\le C \int_{Q^{*}} |(b-\lambda)f_1^1(y_1)|{\rm d}y_1 \int_{\mathbf{R}^n\setminus{Q^{*}}}\frac{1}{|z-y_2|^{2n}}|f_2^2(y_2)| {\rm d}y_2\\
&&\le C \int_{Q^{*}} |(b-\lambda)f_1^1(y_1)|{\rm d}y_1\sum_{k=1}^{\infty} \int_{2^kQ^{*}\setminus{2^{k-1}Q^{*}}}\frac{1}{|z-y_2|^{2n}}|f_2^2(y_2)| {\rm d}y_2\\
&&\le C \|b\|_{{\rm BMO}(\mathbf{R}^n)}\|f_1\|_{L(\log L),Q^{*}}\sum^{\infty}_{k=1} 2^{-kn}\Bigg(\frac{1}{|2^kQ^{*}|}  \int_{2^kQ^{*}} |f_2^2(y_2)| {\rm d}y_2 \Bigg) \\
&&\le C \|b\|_{{\rm BMO}(\mathbf{R}^n)}\|f_1\|_{L(\log L),Q^{*}}\sum^{\infty}_{k=1} 2^{-kn}\|f_2\|_{L(\log L),2^kQ^{*}} \\
&&\le C \|b\|_{{\rm BMO}(\mathbf{R}^n)}\bigg( \frac{1}{|Q^{*}|}  \int_{Q^{*}} |f_1(y_1)|^r {\rm d}y_1\bigg)^{\frac{1}{r}}\sum^{\infty}_{k=1} 2^{-kn}
\Bigg(\frac{1}{|2^kQ^{*}|}  \int_{2^kQ^{*}} |f_2(y_2)|^r {\rm d}y_2 \Bigg)^{\frac{1}{r}} \\
&&\le C \|b\|_{{\rm BMO}(\mathbf{R}^n)} \{\mathcal{M}_{\beta}^1(f_1^r,f_2^r)(x)\}^{\frac{1}{r}}.
\end{eqnarray*}
Based on the above estimates, we obtain
\begin{equation*}
 I_2^2 \le C \|b\|_{{\rm BMO}(\mathbf{R}^n)}\{\mathcal{M}_{\beta}^1(f_1^r,f_2^r)(x)\}^{\frac{1}{r}}.
\end{equation*}
For $I_2^3$, we have
\begin{eqnarray*}
 &&|T((b-\lambda)f_1^2,f_2^1)(z)-T((b-\lambda)f_1^2,f_2^1)(x)|\\
&&\le\int_{\mathbf{R}^{2n}} \frac{C}{(|z-y_1|+|z-y_2|)^{2n}} |(b-\lambda)f_1^2(y_1)||f_2^1(y_2)| {\rm d}y_2{\rm d}y_1   \\
&&\le C \int_{Q^{*}} |f_2^1(y_2)|{\rm d}y_2\sum_{k=1}^{\infty} \int_{2^kQ^{*}\setminus{2^{k-1}Q^{*}}}\frac{|(b-\lambda)f_1^2(y_1)|}{|z-y_1|^{2n}} {\rm d}y_1\\
&&\le C\frac{1}{|Q^{*}|} \int_{Q^{*}} |f_2^1(y_2)|{\rm d}y_2\sum_{k=1}^{\infty} 2^{-kn}\frac{1}{|2^kQ^{*}|} \int_{2^kQ^{*}} |(b-\lambda)f_1^2(y_1)| {\rm d}y_1\\
&&\le C\frac{1}{|Q^{*}|} \int_{Q^{*}} |f_2^1(y_2)|{\rm d}y_2\sum_{k=1}^{\infty} 2^{-kn}\frac{1}{|2^kQ^{*}|} \int_{2^kQ^{*}} |(b-b_{2^kQ^{*}})f_1^2(y_1)| {\rm d}y_1
\end{eqnarray*}
\begin{eqnarray*}
&&\quad+C\frac{1}{|Q^{*}|} \int_{Q^{*}} |f_2^1(y_2)|{\rm d}y_2\sum_{k=1}^{\infty} 2^{-kn}\frac{1}{|2^kQ^{*}|} \int_{2^kQ^{*}} |(b_{2^kQ^{*}}-b_{Q^{*}})f_1^2(y_1)| {\rm d}y_1\\
&&\le C \|b\|_{{\rm BMO}(\mathbf{R}^n)} \|f_2\|_{L(\log L),Q^{*}}\sum^{\infty}_{k=1} 2^{-kn}\|f_1\|_{L(\log L),2^kQ^{*}}\\
&&\quad+C \|b\|_{{\rm BMO}(\mathbf{R}^n)} \|f_2\|_{L(\log L),Q^{*}}\sum^{\infty}_{k=1} 2^{-kn}k\|f_1\|_{L(\log L),2^kQ^{*}}\\
&&\le C \|b\|_{{\rm BMO}(\mathbf{R}^n)}\bigg( \frac{1}{|Q^{*}|}  \int_{Q^{*}} |f_2(y_2)|^r {\rm d}y_2\bigg)^{\frac{1}{r}}\sum^{\infty}_{k=1} 2^{-kn}
\Bigg(\frac{1}{|2^kQ^{*}|}  \int_{2^kQ^{*}} |f_1(y_1)|^r {\rm d}y_1 \Bigg)^{\frac{1}{r}} \\
&&\quad+ C \|b\|_{{\rm BMO}(\mathbf{R}^n)}\bigg( \frac{1}{|Q^{*}|}  \int_{Q^{*}} |f_2(y_2)|^r {\rm d}y_2\bigg)^{\frac{1}{r}}\sum^{\infty}_{k=1} 2^{-kn}k
\Bigg(\frac{1}{|2^kQ^{*}|}  \int_{2^kQ^{*}} |f_1(y_1)|^r {\rm d}y_1 \Bigg)^{\frac{1}{r}} \\
&&\le C \|b\|_{{\rm BMO}(\mathbf{R}^n)}\{\mathcal{M}_{\beta}^2(f_1^r,f_2^r)(x)\}^{\frac{1}{r}}.
\end{eqnarray*}
Finally, we use condition (1.7) to estimate $I_2^4$. Note that for any $x,z\in Q$ and $y_1,y_2\in\mathbf{R}^n\setminus Q^{*}$, $|x-z|\le n \ell(Q)\le \frac{1}{8}\min\{|z-y_1|,|z-y_2|\}$. So
\begin{eqnarray*}
&&|T((b-\lambda)f_1^2,f_2^2)(z)-T((b-\lambda)f_1^2,f_2^2)(x)|\\
&&\le \int_{\mathbf{R}^{2n}}|K(z,y_1,y_2)-K(x,y_1,y_2)| |(b-\lambda)f_1^2(y_1)||f_2^2(y_2)| {\rm d}y_2{\rm d}y_1\\
&&\le C\int_{\mathbf{R}^{n}\setminus Q^{*}} \int_{\mathbf{R}^{n}\setminus Q^{*}}\frac{\ell(Q)^{\gamma}}{(|z-y_1|+|z-y_2|)^{2n+\gamma}} |(b-\lambda)f_1^2(y_1)||f_2^2(y_2)| {\rm d}y_2{\rm d}y_1   \\
&&\le C\sum^{\infty}_{k=1}\int_{2^k\ell(Q)<|z-y_1|+|z-y_2|<2^{k+1}\ell(Q)} \frac{\ell(Q)^{\gamma}~ |(b-\lambda)f_1^2(y_1)|}{(|z-y_1|+|z-y_2|)^{2n+\gamma}} |f_2^2(y_2)| {\rm d}y_2{\rm d}y_1\\
&&\le C \sum^{\infty}_{k=1}\frac{\ell(Q)^{\gamma}}{(2^k\ell(Q))^{2n+\gamma}} \left( \int_{2^{k+2}Q^{*}} |(b-\lambda)f_1^2(y_1)|{\rm d}y_1 \right) \left(\int_{2^{k+2}Q^{*}} |f_2^2(y_2)| {\rm d}y_2 \right)
\end{eqnarray*}
\begin{eqnarray*}
&&\le C \sum^{\infty}_{k=1} 2^{-k\gamma}\left(\frac{1}{|2^{k+2}Q^{*}|} \int_{2^{k+2}Q^{*}} |(b-\lambda)f_1^2(y_1)|{\rm d}y_1 \right)\left(\frac{1}{|2^{k+2}Q^{*}|} \int_{2^{k+2}Q^{*}} |f_2^2(y_2)|{\rm d}y_2 \right)  \\
&&\le C \|b\|_{{\rm BMO}(\mathbf{R}^n)} \mathcal{M}_{L(\log L)}(\vec{f})(x).
\end{eqnarray*}
According to the above estimate, we know that
\begin{equation*}
 I_2^4 \le C \|b\|_{{\rm BMO}(\mathbf{R}^n)} \mathcal{M}_{L(\log L)}(\vec{f})(x).
\end{equation*}
The proof is completed.
\end{proof}

Now, we are  ready to prove Theorem \ref{T2}.
\begin{proof}
We only write out the proof of the boundedness of $T_b^1$, and the other can be got in the same method. By [22, Lemma 6.1], we know that for every $\vec{w}\in A_{\vec{p}}(\mathbf{R}^n)$, there exists a finite constant $1<r_0<\min\{p_1,p_2\}$ such that $\vec{w}\in A_{\vec{p}/r_0}(\mathbf{R}^n)$. From Lemma \ref{lemma3},  for $\vec{w}\in A_{\vec{p}/r_0}(\mathbf{R}^n)$, there exists a $\beta_0>0$ satisfies  that $\sum^2_{i=1}\mathcal{M}_{\beta_0}^i(f_1^{r_0},f_2^{r_0})(x)$ is bounded from
$L^{p_1/r_0}(w_1)(\mathbf{R}^n)\times L^{p_2/r_0}(w_2)(\mathbf{R}^n)$ to $L^{p/r_0}(\nu_{\vec{w}})(\mathbf{R}^n)$.  Hence,
\begin{eqnarray*}
&&\sum^2_{i=1}\|\{\mathcal{M}_{\beta_0}^i(f_1^{r_0},f_2^{r_0})(x)\}^{\frac{1}{r_0}}\|_{L^{p}(\nu_{\vec{w}})}\\
&&=\sum^2_{i=1}\|\{\mathcal{M}_{\beta_0}^i(f_1^{r_0},f_2^{r_0})(x)\}\|^{1/r_0}_{L^{p/r_0}(\nu_{\vec{w}})}\\
&&\le C\|f_1^{r_0}\|^{1/r_0}_{L^{p_1/r_0}(w_1)} \|f_2^{r_0}\|^{1/r_0}_{L^{p_2/r_0}(w_2)}\\
&&=C \|f_1\|_{L^{p_1}(w_1)} \|f_2\|_{L^{p_2}(w_2)}
\end{eqnarray*}
 Because $\nu_{\vec{w}}\in A_{2p}(\mathbf{R}^n)\subset A_{\infty}(\mathbf{R}^n)$, using  inequality (1.8) and Lemma \ref{lemma5},  we obtain
\begin{eqnarray*}
&&\|T_b^1(\vec{f})\|_{L^p(\nu_{\vec{w}})}\le \|M_{\delta}(T_b^1(\vec{f}))\|_{L^p(\nu_{\vec{w}})}
\le C \|
M^{\#}_{\delta}(T_b^1(\vec{f}))\|_{L^p(\nu_{\vec{w}})}   \\
&&\le C \|b\|_{{\rm BMO}(\mathbf{R}^n)} \|\mathcal{M}_{L(\log L)}(\vec{f})(x)+M_{\epsilon}(T(\vec{f}))(x)+ \sum^2_{i=1} \{ \mathcal{M}_{\beta_0}^i(f_1^{r_0},f_2^{r_0})(x) \}^{1/r_0}\|_{L^p(\nu_{\vec{w}})}\\
&&\le C  \|b\|_{{\rm BMO}(\mathbf{R}^n)} \Big( \|\mathcal{M}_{L(\log L)}(\vec{f})(x)\|_{L^p(\nu_{\vec{w}})}+\|M_{\epsilon}(T(\vec{f}))(x)\|_{L^p(\nu_{\vec{w}})}\\
&&\quad\quad+\| \sum^2_{i=1} \{ \mathcal{M}_{\beta_0}^i(f_1^{r_0},f_2^{r_0})(x) \}^{1/r_0}\|_{L^p(\nu_{\vec{w}})} \Big).
\end{eqnarray*}
If we take $\epsilon$ small, we can use  Lemma \ref{lemma4} to obtain
\begin{eqnarray*}
 \|M_{\epsilon}^{\#}(T(\vec{f}))\|_{L^p(\nu_{\vec{w}})}&\le& C \|\mathcal{M}(\vec{f})\|_{L^p(\nu_{\vec{w}})}+C \|\mathcal{M}_1(\vec{f})\|_{L^p(\nu_{\vec{w}})}\\
  &\le& C\|\mathcal{M}_{L(\log L)}(\vec{f})\|_{L^p(\nu_{\vec{w}})}+C\|\mathcal{M}_1(\vec{f})\|_{L^p(\nu_{\vec{w}})}.
\end{eqnarray*}
Now the desired result follows from Lemma \ref{lemma1} and Lemma \ref{lemma2}.

In the above proof, we  note that when we use the inequality (1.8) we need to explain that $\| M_{\epsilon}(T(\vec{f}))\|_{L^p(\nu_{\vec{w}})}$ and  $\|M_{\delta}(T_b^1(\vec{f}))\|_{L^p(\nu_{\vec{w}})}$ are finite. A detailed proof was given in page 33 of \cite{Lerner}, and the proof  can also be applied to here owing to the boundedness of $T$ which was proved in [20, Theorem 2] . The reader can see \cite{Lerner} and \cite{Hu4}.
\end{proof}

\section{Proof of Theorem 1}
The idea of considering truncated operators to prove compactness results in the linear setting  can trace back to \cite{Krantz}, and this method was adopted in \cite{Clop}. Recently, B\'enyi  et al. (see \cite{Benyi1})  introduced a new smooth truncation to simplify the computations. We will use this technique to prove Theorem \ref{T1}.

Let $\varphi=\varphi(x,y_1,y_2)$ be a non-negative function in $C_c^{\infty}(\mathbf{R}^{3n})$, and it satisfy ${\rm supp}\,\varphi\subset\{(x,y_1,y_2):\max(|x|,|y_1|,|y_2|)<1\} $, $\int_{\mathbf{R}^{3n}} \varphi(u)du=1$.  For $\delta>0$, let $\chi^{\delta}=\chi^{\delta}(x,y_1,y_2)$ be the characteristic function of the set $\{(x,y_1,y_2):\max(|x-y_1|,|x-y_2|)\ge 3\delta/2\}$, and let
\begin{eqnarray*}
\psi^{\delta}=\varphi_{\delta}*\chi^{\delta},
\end{eqnarray*}
where $\varphi_{\delta}(x,y_1,y_2)=(\delta/4)^{-3n}\varphi(4x/\delta,4y_1/\delta,4y_2/\delta)$. By an easy  calculation, we get that $\psi^{\delta} \in C^{\infty}(\mathbf{R}^{3n})$, $\|\psi^{\delta}\|_{L^{\infty}}\le1$,
\begin{eqnarray*}
{\rm supp}\,\psi^{\delta}\subset\{(x,y_1,y_2):\max(|x-y_1|,|x-y_2|)\ge\delta\},
\end{eqnarray*}
and $\psi^{\delta}(x,y_1,y_2)=1$ if $\max(|x-y_1|,|x-y_2|)\ge 2\delta$.

We define the truncated kernel
\begin{eqnarray*}
K^{\delta}(x,y_1,y_2)=\psi^{\delta}(x,y_1,y_2)K(x,y_1,y_2),
\end{eqnarray*}
where $K(x,y_1,y_2)$ is the kernel associated to the  bilinear  singular integral  operator $T$  considered in Theorem \ref{T1}. It's easy to verify that $K^{\delta}$ also satisfies condition (1.2) and (1.7). Denote by $T^{\delta}$  the bilinear operator that associated with  kernel $K^{\delta}$ in the sense of  (1.1). The following Lemma was proved in \cite{Benyi1}:
\begin{lemma}\label{lemma6}
For all $x\in \mathbf{R}^n$, $b,b_1,b_2\in C^{\infty}_c(\mathbf{R}^n)$, if $\vec{w}=(w_1,w_2)\in A_{\vec{p}}(\mathbf{R}^{2n})$, then
\begin{eqnarray*}
&&\lim_{\delta\to0}\|[b,T^{\delta}]_1-[b,T]_1\|_{L^{p_1}(w_1)\times L^{p_2}(w_2)\to L^p(\nu_{\vec{w}})}=0, \\
&&\lim_{\delta\to0}\|[b,T^{\delta}]_2-[b,T]_2\|_{L^{p_1}(w_1)\times L^{p_2}(w_2)\to L^p(\nu_{\vec{w}})}=0, \\
&&\lim_{\delta\to0}\|[b_2,[b_1,T^{\delta}]_1]_2-[b_2,[b_1,T]_1]_2\|_{L^{p_1}(w_1)\times L^{p_2}(w_2)\to L^p(\nu_{\vec{w}})}=0.
\end{eqnarray*}
\end{lemma}
By the size condition (1.2), Lemma \ref{lemma6} can be proved by the argument used in \cite{Benyi1}.
\begin{lemma}\label{lemma7}
Suppose that $T$ is as in Theorem \ref{T1}. Then, for all $\zeta>0$, there  exists a positive constant $C$ such that for all $\vec{f}$ in the product of $L^{p_j}(\mathbf{R}^n)$ with $1\le p_j<\infty$ and all $x\in\mathbf{R}^n$
\begin{eqnarray*}
T^{*}(\vec{f})(x)\le C\big(M_{\zeta}(T(\vec{f})) (x)\big)+\sum^2_{i=1}\mathcal{M}_{2,i}(\vec{f})(x)+\mathcal{M}(\vec{f})(x),
\end{eqnarray*}
where  $T^{*}(\vec{f})$ is the maximal truncated bilinear singular integral operator defined as
\begin{eqnarray*}
T^{*}(f_1,f_2)=\sup_{\eta>0}\bigg|\int\int_{\max(|x-y_1|,|x-y_2|)>\eta} K(x,y_1,y_2)f_1(y_1)f_2(y_2){\rm d}y_1{\rm d}y_2\bigg|.
\end{eqnarray*}
\end{lemma}
   The proof of the Lemma \ref{lemma7} is similar to the proof of [15, Theorem 1], so we leave it to the interested reader.
\begin{lemma}\label{lemma8}
Let $1<p<\infty$, $w\in A_p(\mathbf{R}^n)$ and $\mathcal{H}\subset L^p(w)$.  If
\begin{itemize}
\item[\rm(i)] $\mathcal{H}$ is bounded in $L^p(w)$;
\item[\rm(ii)]  $\lim\limits_{A\to\infty} \int_{|x|>A} |f(x)|^p w(x){\rm d}x=0$ uniformly for $f\in \mathcal{H}$;
\item[\rm(iii)] $\lim\limits_{t\to0} \|f(\cdot+t)-f(\cdot)\|_{L^p(w)}=0$ uniformly for $f\in \mathcal{H}$.
\end{itemize}
then   $\mathcal{H}$ is precompact in $L^p(w)$.
\end{lemma}
This Lemma was given in \cite{Clop}.

 Now, we are ready to prove Theorem \ref{T1}.
\begin{proof}
We will work with the commutator $[b,T]_1$ first, and the proof of commutator $[b,T]_2$ can be get by symmetry. From Lemma \ref{lemma6}, we only need to  prove the compactness of  $[b,T^{\delta}]_1$ for any fixed $\delta\le1/8$.
 By Theorem \ref{T2}, it suffices to show the result for $b\in C^{\infty}_c(\mathbf{R}^n)$. Suppose $f_1,f_2$ belong to
\begin{eqnarray*}
B_1(L^{p_1}(w_1))\times B_1(L^{p_2}(w_2))=\{(f_1,f_2):\|f_1\|_{L^{p_1}(w_1)},\|f_2\|_{L^{p_2}(w_2)}\le1\},
\end{eqnarray*}
where $\vec{w}\in A_{\vec{p}}$. We need to prove that the following three conditions hold:
\begin{itemize}
\item[\rm(a)]  $[b,T^{\delta}]_1(B_1(L^{p_1}(w_1))\times B_1(L^{p_2}(w_2)))$ is bounded in $L^p(\nu_{\vec{w}})$;
\item[\rm(b)] $\lim\limits_{A\to\infty} \int_{|x|>A} |[b,T^{\delta}]_1(f_1,f_2)(x)|^p \nu_{\vec{w}}(x){\rm d}x=0$;
\item[\rm(c)] Given $0<\xi<1/8$, there exists a sufficiently small $t_0 (t_0=t_0(\xi))$ such that for all $0<|t|<t_0$, we have
\begin{eqnarray}
\|[b,T^{\delta}]_1(f_1,f_2)(\cdot)-[b,T^{\delta}]_1(f_1,f_2)(\cdot+t)\|_{L^p(\nu_{\vec{w}})} \le C\xi.
\end{eqnarray}
\end{itemize}
  It is easy to find that the condition (a) holds because of the boundedness of $[b,T^{\delta}]_1 $ in Theorem \ref{T2}. Now, we prove the condition (b) using some ideas in \cite{Hu1}. Let $R>0$ be large enough such that  ${\rm supp}\,b\subset B(0,R)$ and let $A\ge \max(2R,1)$, $l$ be a nonnegative integer. For any $|x|>A$, denote
\begin{eqnarray*}
&&V_R^0(x)=\int_{|y_2|\le|x|}\int_{|y_1|\le R} |K^{\delta}(x,y_1,y_2)|\prod^2_{j=1}|f_j(y_j)|{\rm d}y_1{\rm d}y_2,\\
&&V_R^l(x)=\int_{2^{l-1}|x|\le|y_2|\le2^l|x|}\int_{|y_1|\le R} |K^{\delta}(x,y_1,y_2)|\prod^2_{j=1}|f_j(y_j)|{\rm d}y_1{\rm d}y_2,
\end{eqnarray*}
when $l>0$. From condition (1.2), we deduce that
\begin{eqnarray*}
V_R^l(x)&\le&C\int_{2^{l-1}|x|\le|y_2|\le2^l|x|}\int_{|y_1|\le R}\frac{1}{ (|x-y_1|+|x-y_2|)^{2n}}|f_1(y_1)||f_2(y_2)|{\rm d}y_1{\rm d}y_2\\
&\le&C\int_{2^{l-1}|x|\le|y_2|\le2^l|x|}\int_{|y_1|\le R} \frac{|f_1(y_1)||f_2(y_2)|}{(|x|+|x-y_2|)^{2n}}{\rm d}y_1{\rm d}y_2\\
&\le&C\frac{1}{(2^{l-1}|x|)^{2n}}\int_{2^{l-1}|x|\le|y_2|\le2^l|x|}\int_{|y_1|\le R}|f_1(y_1)||f_2(y_2)|{\rm d}y_1{\rm d}y_2\\
&\le&C\frac{1}{(2^{l-1}|x|)^{2n}}\bigg(\int_{B(0,R)} w_1^{-\frac{1}{p_1-1}}(y_1)dy_1\bigg)^{1-1/p_1}\bigg(\int_{B(0,2^l|x|)} w_2^{-\frac{1}{p_2-1}}(y_2)dy_2\bigg)^{1-1/p_2}.
\end{eqnarray*}
The same estimate can be got for $V_R^0(x)$. Note that $w_1^{-\frac{1}{p_1-1}}\in A_{\infty}(\mathbf{R}^n)$, so there exists a constant  $\theta_1\in(0,1)$ such that
\begin{eqnarray*}
\int_{B(0,R)} w_1^{-\frac{1}{p_1-1}}(y_1){\rm d}y_1\le C(2^{-(j+l)}RA^{-1})^{n\theta_1}\int_{B(0,2^{l+j}A)} w_1^{-\frac{1}{p_1-1}}(y_1){\rm d}y_1.
\end{eqnarray*}
Since $p>1$, it follows that
\begin{eqnarray*}
&&\bigg(\int_{2^{j-1}A\le|x|\le2^jA} |[b,T^{\delta}]_1(f_1,f_2)(x)|^p \nu_{\vec{w}}(x){\rm d}x\bigg)^{1/p} \\
&&~\le C\sum^{\infty}_{l=0} \bigg(\int_{2^{j-1}A\le|x|\le2^jA} |V_R^l(x)|^p \nu_{\vec{w}}(x){\rm d}x\bigg)^{1/p}\\
&&~\le C\sum^{\infty}_{l=0} \bigg(\int_{2^{j-1}A\le|x|\le2^jA} \frac{1}{(2^{l-1}|x|)^{2np} } \nu_{\vec{w}}(x){\rm d}x\bigg)^{1/p} \\
&&\quad\times\bigg(\int_{B(0,R)} w_1^{-\frac{1}{p_1-1}}(y_1){\rm d}y_1\bigg)^{1-1/p_1}\bigg(\int_{B(0,2^{l+j}A)} w_2^{-\frac{1}{p_2-1}}(y_2){\rm d}y_2\bigg)^{1-1/p_2}\\
&&~\le C\sum^{\infty}_{l=0}(2^{l+j-2}A)^{-2n}(2^{-(j+l)}RA^{-1})^{n\theta_1(1-1/p_1)} \bigg(\int_{B(0,2^jA)}\nu_{\vec{w}}(x){\rm d}x\bigg)^{1/p} \\
&&\quad\times\bigg(\int_{B(0,2^{l+j}A)} w_1^{-\frac{1}{p_1-1}}(y_1){\rm d}y_1\bigg)^{1-1/p_1}\bigg(\int_{B(0,2^{l+j}A)} w_2^{-\frac{1}{p_2-1}}(y_2){\rm d}y_2\bigg)^{1-1/p_2}
\end{eqnarray*}
\begin{eqnarray*}
&&~\le C\sum^{\infty}_{l=0}(2^{l+j}A)^{-2n}(2^{-(j+l)}RA^{-1})^{n\theta_1(1-1/p_1)}(2^{j+l}A)^{2n}\\
&&~\le C\sum^{\infty}_{l=0} 2^{l(-n\theta_1(1-1/p_1))}2^{j(-n\theta_1(1-1/p_1))}(R/A)^{n\theta_1(1-1/p_1)}\\
&&~\le C2^{j(-n\theta_1(1-1/p_1))}(R/A)^{n\theta_1(1-1/p_1)}.
\end{eqnarray*}
Thus, it is easy to see,
\begin{eqnarray*}
\bigg(\int_{|x|>A} |[b,T^{\delta}]_1(f_1,f_2)(x)|^p \nu_{\vec{w}}(x){\rm d}x\bigg)^{1/p}\le C(R/A)^{n\theta_1(1-1/p_1)} \to 0.
\end{eqnarray*}
as  $A\to \infty$.

So, it suffices to  verify condition (c). To prove (4.3), we  decompose the expression inside the $L^p(\nu_{\vec{w}})$ norm as follows:
\begin{eqnarray*}
&&[b,T^{\delta}]_1(f_1,f_2)(x)-[b,T^{\delta}]_1(f_1,f_2)(x+t)      \\
&&=\int\int_{\min(|x-y_1|,|x-y_2|)>\eta} K^{\delta}(x,y_1,y_2)(b(x+t)-b(x))\prod^2_{j=1}f_j(y_j){\rm d}\vec{y}         \\
&&\quad+\int\int_{\min(|x-y_1|,|x-y_2|)>\eta} (K^{\delta}(x,y_1,y_2)-K^{\delta}(x+t,y_1,y_2))(b(y_1)-b(x+t))\prod^2_{j=1}f_j(y_j){\rm d}\vec{y}   \\
&&\quad+\int\int_{\min(|x-y_1|,|x-y_2|)<\eta} K^{\delta}(x,y_1,y_2)(b(y_1)-b(x))\prod^2_{j=1}f_j(y_j){\rm d}\vec{y}\\
&&\quad+\int\int_{\min(|x-y_1|,|x-y_2|)<\eta} K^{\delta}(x+t,y_1,y_2)(b(x+t)-b(y_1))\prod^2_{j=1}f_j(y_j){\rm d}\vec{y}\\
&&=A(x)+B(x)+C(x)+D(x),
\end{eqnarray*}
where $0<\eta<1$ and the choice of $\eta$ will be specified later.

Now we denote
\begin{eqnarray*}
&&E=\{(x,y_1,y_2): \min(|x-y_1|,|x-y_2|)>\eta\},\\
&&F=\{(x,y_1,y_2): \max(|x-y_1|,|x-y_2|)>2\delta\},\\
&&G=\{(x,y_1,y_2): \max(|x-y_1|,|x-y_2|)>\eta\},\\
&&H=\{(x,y_1,y_2): \delta<\max(|x-y_1|,|x-y_2|)<2\delta\}.
\end{eqnarray*}
It is obvious that $K^{\delta}(x,y_1,y_2)=K(x,y_1,y_2)$ on $F$. Consequently,
\begin{eqnarray*}
&&\bigg|\int\int_E K^{\delta}(x,y_1,y_2)f_1(y_1)f_2(y_2){\rm d}y_1{\rm d}y_2-\int\int_G K(x,y_1,y_2)f_1(y_1)f_2(y_2){\rm d}y_1{\rm d}y_2\bigg|\\
&&=\bigg|\int\int_{(E\cap F)\cup (E\cap H)} K^{\delta}(x,y_1,y_2)f_1(y_1)f_2(y_2){\rm d}y_1{\rm d}y_2  \\
&&\quad-\int\int_{(E\cap F)\cup (G\setminus (E\cap F))} K(x,y_1,y_2)f_1(y_1)f_2(y_2){\rm d}y_1{\rm d}y_2\bigg|
\end{eqnarray*}
\begin{eqnarray*}
&&\le\bigg|\int\int_{E\cap H} K^{\delta}(x,y_1,y_2)f_1(y_1)f_2(y_2){\rm d}y_1{\rm d}y_2\bigg|\\
&&\quad+\int\int_{G\cap E^c} |K(x,y_1,y_2)f_1(y_1)f_2(y_2)|{\rm d}y_1{\rm d}y_2\\
&&\quad+\int\int_{G\cap F^c\cap E} |K(x,y_1,y_2)f_1(y_1)f_2(y_2)|{\rm d}y_1{\rm d}y_2\\
&&\quad+\int\int_{G\cap F^c\cap E^c} |K(x,y_1,y_2)f_1(y_1)f_2(y_2)|{\rm d}y_1{\rm d}y_2.
\end{eqnarray*}
Now, we  estimate the above four parts using condition (1.2),
\begin{eqnarray*}
&&\bigg|\int\int_{E\cap H} K^{\delta}(x,y_1,y_2)f_1(y_1)f_2(y_2){\rm d}y_1{\rm d}y_2 \bigg|\\
&&\le\int\int_{H} \frac{|f_1(y_1)||f_2(y_2)|}{(|x-y_1|+|x-y_2|)^{2n}}{\rm d}y_1{\rm d}y_2 \\
&&\le C\mathcal{M}(f_1,f_2)(x).\\
&&\int\int_{G\cap E^c} |K(x,y_1,y_2)f_1(y_1)f_2(y_2)|{\rm d}y_1{\rm d}y_2 \\
&&\le\int_{|x-y_1|<\eta}\int_{|x-y_2|>\eta} \frac{|f_1(y_1)||f_2(y_2)|}{(|x-y_1|+|x-y_2|)^{2n}}{\rm d}y_1{\rm d}y_2\\
&&\le \int_{|x-y_1|<\eta}|f_1(y_1)|{\rm d}y_1\sum^{\infty}_{k=1}\int_{2^{k-1}\eta<|x-y_2|<2^k\eta}\frac{|f_2(y_2)|}{|x-y_2|^{2n}}{\rm d}y_2\\
&&\le C \sum^{\infty}_{k=1} 2^{-kn}\frac{1}{|B(x,\eta)|}\int_{B(x,\eta)}|f_1(y_1)|{\rm d}y_1\frac{1}{|B(x,2^k\eta)|}\int_{B(x,2^k\eta)}|f_2(y_2)|{\rm d}y_2\\
&&\le C\sum^2_{i=1}\mathcal{M}_{2,i}(f_1,f_2)(x),
\end{eqnarray*}
where the set $G\cap E^c$ includes two cases: $\{(x,y_1,y_2):|x-y_1|<\eta,|x-y_2|>\eta\}$ and $\{(x,y_1,y_2):|x-y_1|>\eta,|x-y_2|<\eta\}$. Since the estimates on these two regions are similar, we omit the late one. This method will be used several times in the following.

Because $\eta<|x-y_1|<2\delta$, $\eta<|x-y_2|<2\delta$ when $(x,y_1,y_2)\in G\cap F^c\cap E$. Hence,
\begin{eqnarray*}
&&\int\int_{G\cap F^c\cap E} |K(x,y_1,y_2)f_1(y_1)f_2(y_2)|{\rm d}y_1{\rm d}y_2\\
&&\le4\delta\int\int_{G}\frac{|f_1(y_1)||f_2(y_2)|}{(|x-y_1|+|x-y_2|)^{2n+1}}{\rm d}y_1{\rm d}y_2\le C \frac{\delta}{\eta}\mathcal{M}(f_1,f_2)(x),\\
&&\int\int_{G\cap F^c\cap E^c} |K(x,y_1,y_2)f_1(y_1)f_2(y_2)|{\rm d}y_1{\rm d}y_2\\
&&\le\int_{|x-y_1|<\eta}\int_{|x-y_2|>\eta}\frac{|f_1(y_1)||f_2(y_2)|}{(|x-y_1|+|x-y_2|)^{2n}}{\rm d}y_1{\rm d}y_2\le C \sum^2_{i=1}\mathcal{M}_{2,i}(f_1,f_2)(x).
\end{eqnarray*}
In summary, we  get
\begin{eqnarray*}
|A(x)|&\le& C|t| \|\nabla b\|_{L^{\infty}}  \bigg|\int\int_{E} K^{\delta}(x,y_1,y_2)f_1(y_1)f_2(y_2){\rm d}y_1{\rm d}y_2\bigg| \\
&\le& C|t| \|\nabla b\|_{L^{\infty}} \bigg( T^{*}(f_1,f_2)(x) +\frac{1}{\eta}\mathcal{M}(f_1,f_2)(x)+\sum^2_{i=1}\mathcal{M}_{2,i}(f_1,f_2)(x) \bigg).
\end{eqnarray*}
 From Lemma \ref{lemma2}, Lemma \ref{lemma7} and  [22, Theorem 3.7],  we obtain
\begin{equation}
 \|A\|_{L^p(\nu_{\vec{w}})}\le C|t|(1+1/\eta).
\end{equation}
In order to estimate $B(x)$, by a  consequence of condition (1.7), we have
\begin{eqnarray*}
|K(x,y_1,y_2)-K(x^{'},y_1,y_2)|
   \le \frac{D|x-x^{'}|^{\gamma}}{(|x-y_1|+|x-y_2|)^{2n+
  \gamma }}
\end{eqnarray*}
when $|x-x^{'}|\le\frac{1}{8}\min{|x-y_1|,|x-y_2|}$. Then
\begin{eqnarray*}
|B(x)|&\le&C\|b\|_{L^{\infty}}\int\int_E|K^{\delta}(x,y_1,y_2)-K^{\delta}(x+t,y_1,y_2)||f_1(y_1)||f_2(y_2)|{\rm d}y_1{\rm d}y_2 \\
&\le& C\|b\|_{L^{\infty}}|t|^{\gamma}\int\int_G\frac{|f_1(y_1)||f_2(y_2)|}{(|x-y_1|+|x-y_2|)^{2n+\gamma}}{\rm d}y_1{\rm d}y_2 \\
&\le&C\|b\|_{L^{\infty}}\frac{|t|^{\gamma}}{\eta^{\gamma}}\mathcal{M}(f_1,f_2)(x).
\end{eqnarray*}
Therefore,
\begin{equation}
 \|B\|_{L^p(\nu_{\vec{w}})}\le C\frac{|t|^{\gamma}}{\eta^{\gamma}}.
\end{equation}
 For any $0<\beta<1$, we have $|b(x)-b(y_1)|\le|x-y_1|^{\beta}$.
  Hence, using the size condition (1.2) and the property of the support of $K^{\delta}(x,y_1,y_2)$, we can  estimate the third term:
\begin{eqnarray*}
|C(x)|&\le& C\|\nabla b\|_{L^{\infty}}\eta\int_{{|x-y_1|<\eta}}\int_{{|x-y_2|>\eta}}\frac{|f_1(y_1)||f_2(y_2)|}{(|x-y_1|+|x-y_2|)^{2n}}{\rm d}y_1{\rm d}y_2\\
&&+ C\int_{{|x-y_1|>\eta}}\int_{{|x-y_2|<\eta}}\frac{|f_1(y_1)||f_2(y_2)|}{(|x-y_1|+|x-y_2|)^{2n-\beta}}{\rm d}y_1{\rm d}y_2\\
&\le&C \eta \mathcal{M}_{2,1}(f_1,f_2)(x)\\
&&+C\int_{|x-y_2|<\eta}|f_2(y_2)|{\rm d}y_2\sum^{\infty}_{k=1}\int_{2^{k-1}\eta<|x-y_1|<2^k\eta}\frac{|f_1(y_1)|}{|x-y_1|^{2n-\beta}}{\rm d}y_1\\
&\le&C ( \eta \mathcal{M}_{2,1}(f_1,f_2)(x)+\eta^{\beta} \mathcal{M}_{\beta}^2(f_1,f_2)(x)),
\end{eqnarray*}
provided $\eta<\delta$. From Lemma \ref{lemma3}, we know that
\begin{equation}
 \|C\|_{L^p(\nu_{\vec{w}})}\le C\eta,
\end{equation}
when we take sufficiently small $\beta$.

Finally, for the last part we proceed in a similar way,  by  replacing $x$ with $x+t$ and the region of integration $E^c$ with a larger one  $\{(x,y_1,y_2): \min(|x+t-y_1|,|x+t-y_2|)<\eta+|t|\}$.  By the fact that $x\in B(x+t,\eta+|t|)$, where $B(x+t,\eta+|t|)$ denote the ball centered at $x+t$ and with radius $\eta+|t|$, we obtain
\begin{equation}
 \|D\|_{L^p(\nu_{\vec{w}})}\le C(|t|+\eta).
\end{equation}
Let us now define $t_0=\xi^2$ and for each $0<|t|<t_0$, choose $\eta=|t|/\xi$. Then inequalities (4.2)-(4.5) imply (4.1), and in this way, we can conclude  that $[b,T]_1$ is compact.  By symmetry,  $[b,T]_2$ is also compact.
\end{proof}

\begin{acknowledgement}
  This research was supported  by the NNSF of China (Grant No. 11271330).
\end{acknowledgement}

\vspace{\baselineskip}

\end{document}